\documentclass[11pt, numbers=noenddot, headinclude,
xcolor=dvipsnames, a4paper]{scrartcl}
\usepackage{etex}
\reserveinserts{32}

\usepackage{amsmath}
\usepackage{amsfonts}
\usepackage{amssymb}
\usepackage{units}
\usepackage{enumitem}
\usepackage[arrow, matrix, curve]{xy}

\usepackage[standard, thmmarks, amsmath, hyperref]{ntheorem}
\newtheorem{lem}{Lemma}[section]

\newtheorem{thm}[lem]{Theorem}
\newtheorem{cor}[lem]{Corollary}
\newtheorem{fact}[lem]{Fact}
\theorembodyfont{\upshape}
\newtheorem{rem}[lem]{Remark}
\theorembodyfont{\upshape}
\newtheorem{defi}[lem]{Definition}

\newcommand{\demphmath}[1]{#1}
\theoremprework{\setlength{\theorempreskipamount}{.5em plus.5em
minus.5em}\setlength{\theorempostskipamount}{.5em plus.5em minus.5em}}

\theorembodyfont{\upshape}

\DeclareMathOperator*{\Aut}{Aut}
\DeclareMathOperator*{\fin}{fin}
\DeclareMathOperator*{\dcl}{dcl}
\DeclareMathOperator*{\acl}{acl}

\DeclareMathOperator*{\Isom}{Isom}

\DeclareMathOperator*{\im}{im}

\DeclareMathOperator*{\tp}{tp}

\newcommand{\U}{\mathbb{U}}
\newcommand{\N}{\mathbb{N}}

\newcommand{\Z}{\mathbb{Z}}

\newcommand{\F}{\mathbb{F}}

\def\Ind#1#2{#1\setbox0=\hbox{$#1x$}\kern\wd0\hbox to 0pt{\hss$#1\mid$\hss}
\lower.9\ht0\hbox to 0pt{\hss$#1\smile$\hss}\kern\wd0}
\def\Notind#1#2{#1\setbox0=\hbox{$#1x$}\kern\wd0\hbox to
0pt{\mathchardef\nn="0236\hss$#1\nn$\kern1.4\wd0\hss}\hbox 
to 0pt{\hss$#1\mid$\hss}\lower.9\ht0
\hbox to 0pt{\hss$#1\smile$\hss}\kern\wd0}
\def\ind{\mathop{\mathpalette\Ind{}}}

\usepackage{gensymb}

\usepackage[utf8]{inputenc}
\usepackage[english]{babel}
\usepackage{hyperref}
\usepackage{graphicx}
\usepackage{nameref}
\usepackage{epigraph}
\usepackage[colorinlistoftodos]{todonotes}
\usepackage{xspace}
\usepackage{xcolor}
\newcommand{\demph}[1]{\emph{#1}}

\usepackage{xcolor}
\hypersetup{
    colorlinks,
    citecolor=green!75!black,
    filecolor=black,
    linkcolor=blue!95!black,
    urlcolor=black
}

\color{blue!20!black}
\usepackage{scrpage2}

\usepackage{makeidx}

\usepackage{tabularx}
\usepackage{textcomp}
\usepackage{relsize}
\usepackage{tikz}

 \newcommand{\amalgam}{\otimes}

\newcommand{\amalgamo}[1]{*_{#1}}
 
 \newcommand{\samalgamo}[3]{\langle #1#2\rangle\amalgam_{#2}\langle #3#2\rangle}
\newcommand{\struc}[1]{\langle #1\rangle}
\newcommand{\kat}{Kat\v{e}tov\xspace}
\newcommand{\SIR}{SI-structure\xspace}
\newcommand{\SIRs}{SI-structures\xspace}
\newcommand{\Fraisse}{Fra\"issé\xspace}
\DeclareMathOperator*{\Fr}{\mathrm{Fr}}
\newcommand{\Komega}{\mathcal{K}_{\omega}}

\newtheorem*{thmnonumber}{Theorem}

\usepackage{makeidx}
\makeindex
\title{\Fraisse Structures with Universal Automorphism Groups}
\author{Isabel Müller}
\begin{document}

 \maketitle
\begin{abstract}
We prove that the automorphism group of a \Fraisse structure $M$ equipped with
a notion of stationary
independence is universal for the class of automorphism groups of substructures
of $M$. Furthermore, we show that this applies to certain homogeneous
$n$-gons. 
\end{abstract}

\section{Introduction}
Certain homogeneous structures are universal with respect to the class of
their substructures: The Rado graph is
universal for the class of all countable graphs, the rationals as a dense linear
order for the class of all countable linear orders and Urysohn's
universal Polish
space for the class of all Polish spaces. Jaligot asked whether a
universal structure $M$ transfers its universality onto its automorphism
group, i.e. whether $\Aut(M)$ is universal for the
class of automorphism groups of substructures of $M$ (cf. \cite{Jaligot2007}).
Recently, Doucha showed that, for an uncountable structure $M$, the answer 
to Jaligot's question is rarely positive (\cite{Doucha}). Countable homogeneous 
structures on the contrary, most often have universal automorphism 
groups. In fact, the only known counter example was pointed out by 
Piotr Kowalski and is given by the \Fraisse limit of finite fields in fixed 
characteristic $p$, which coincides with the algebraic closure of 
$\F_p$. Its automorphism group is $\hat{\Z}$, which 
is torsion free and hence does not embed any automorphism group of a finite 
field. It is still unknown if there is a relational countable counterexample.

We will prove that in the case
where M is a \Fraisse structure admitting a certain stationary
independence relation, the automorphism group $\Aut(M)$ will be universal for 
the class of automorphism groups of substructures of $M$. 

Uspenskij \cite{Uspenskij2001}, using a careful
construction of Urysohn's universal Polish space given by \kat
\cite{Katetov1988}, proved that its isometry
group is universal for the class of all Polish groups, which corresponds
to the class of isometry groups of Polish spaces \cite{Gao2003}.
The idea of
\kat thereby can be described as follows: Given a Polish
space $X$, he constructed a new metric space $E_1(X)$ consisting of $X$ together
with all possible 1-point metric extensions, while assigning the smallest
possible distance between new points. Under minor
restrictions, the space obtained is again Polish, denoted by the
first \kat space of $X$. Iterating this, i.e. building one \kat space
over the other, he constructed a copy of Urysohn's space
itself. Furthermore, all isometries of $X$ extend in a unique way at every step
of the construction, which yields the desired embedding of $\Isom(X)$ into
$\Isom(\U)$. 

In \cite{Bilge2012}
Bilge adapted this
construction to \Fraisse limits of rational structures with free
amalgamation by gluing extensions freely over the given space. Both Urysohn's 
spaces and \Fraisse classes with free amalgamation carry an
independence relation as introduced by Tent and Ziegler
\cite{Tent2012}.
In this paper, we will show that the mere presence of a stationary
independence relation within a \Fraisse structure $M$ allows us to
mimic \kat's construction of Urysohn's universal metric space, starting with any
structure $X$ embeddable in $M$. With the help of the given independence
relation, we will
glue "small" extensions of $X$ independently
and construct
an analog of \kat spaces in the non-metric setting, thereby ensuring
that
the automorphisms of $X$ extend canonically to its \kat spaces and that the
extensions behave well under composition. In particular, we will give a
positive answer to the question of
Jaligot for the class
of \Fraisse limits with stationary independence relation by proving the
following result (Theorem \ref{mainthm}):

\begin{thmnonumber}
Let $M$ be a \Fraisse structure with stationary independence relation and
$\Komega$ the class of all countable structures
embeddable into
$M$. Then the automorphism group $\Aut(M)$ is universal for the class
$\Aut(\Komega):=\{\Aut(X)\mid X\in\Komega\}$, i.e. every group in
$\Aut(\Komega)$ can be continuously embedded as a subgroup into $\Aut(M)$.  
\end{thmnonumber}
Note, that every automorphism group of a countable first order structure $M$ 
can be considered as a Polish group if we equip it with the topology of 
pointwise convergence. The basic open sets for that topology
\begin{equation*}
 \mathcal{O}_u:=\{f\in\Aut(M)\mid f|_A=u\}
\end{equation*}
 
are determined by finite partial isomorphisms $u:A\rightarrow M$, where 
$A\subseteq M$ is a finite subset of $M$.\\

\textbf{Acknowledgements.} I want to thank Amador Martin-Pizarro for guiding me
into the topic and taking so much time to listen to first ideas and giving many
helpful comments all along the way. I also want to thank Andreas Baudisch for
fruitful
discussions during my diploma thesis and Katrin Tent for suggesting the
homogeneous generalized $n$-gons as a possible example and for comments about an
earlier version of the paper. 
\section{Preliminaries}
Let us briefly recall the central concepts of \Fraisse theory used in the
article.  For further reading and proofs in this topic, there is a plethora of
sources, see for example \cite[p.
158ff.]{Hodges} or \cite[p. 69ff.]{Tent2012a}.  

Let $L$ be a countable language and $\mathcal{K}$ a class
of finitely
 generated $L$-structures which is countable up to isomorphism
types. We call $\mathcal{K}$ a \demph{\Fraisse class} if the following
three conditions are satisfied:
\begin{itemize}
 \item [HP] For any finitely generated $L$-structure $A$ which is embeddable
into some $B\in \mathcal{K}$, there is a structure $A'$ in $\mathcal{K}$
isomorphic to
$A$.
\item [JEP] For every $B$ and $C$ in $\mathcal{K}$, there is some
$D\in\mathcal{K}$ such that both $B$ and $C$ are embeddable into $D$. 
\item [AP] For every $A,B$ and $C$ in $\mathcal{K}$ together with embeddings 
$f_1:A\rightarrow B$ and $f_2:A\rightarrow C$, there are some $D$ in
$\mathcal{K}$ and embeddings $g_1:B\rightarrow D$ and $g_2:C\rightarrow D$ such
that the following diagram commutes:
$$
\begin{xy}
 \xymatrixcolsep{3pc}\xymatrixrowsep{1pc}
 \xymatrix{
& B \ar@{-->}^{g_1}[rd]& \\
A \ar@{->}^{f_1}[ru] \ar@{->}_{f_2}[rd]& & D.\\
& C \ar@{-->}_{g_2}[ru] &\\
 }
\end{xy}
$$
\end{itemize}

We call the class of all finitely generated substructures of an $L$-structure 
$M$ the \demph{skeleton} of $M$. An $L$-structure $M$
is called \demph{rich} with respect to a class $\mathcal{K}$ of
finitely generated $L$-structures
if, for all
$A$ and $B$ in $\mathcal{K}$ together with embeddings $f:A\rightarrow B$ and
$g:A\rightarrow M$, there is an embedding $h:B\rightarrow M$ such that 
$h\circ f=g$.
Finally, an $L$-structure is
\demphmath{$\mathcal{K}$}\demph{-saturated} if its skeleton is exactly
$\mathcal{K}$ and it is furthermore rich with respect to $\mathcal{K}$.

The following fact is the main theorem of \Fraisse theory:
\begin{fact}\label{Fraissethm}
Assume $\mathcal{K}$ to be a class of finitely generated $L$-structures,
countable up to isomorphism types. Then
there is a countable $\mathcal{K}$-saturated structure $M$  if and only if
$\mathcal{K}$ is a \Fraisse class. Furthermore, any two
countable $\mathcal{K}$-saturated structures are isomorphic.
\end{fact}
Given a \Fraisse-class $\mathcal{K}$, the corresponding $\mathcal{K}$-saturated
structure as above is called the \demph{\Fraisse limit}
of $\mathcal{K}$. Note that a countable
structure $M$ is the \Fraisse limit
of its
skeleton if and only if $M$ is \demph{homogeneous}, i.e. every partial
isomorphism between finitely generated substructures can be extended to an
automorphism of $M$. We call such structures \demph{\Fraisse structures}.

\section{Stationary Independence}
The main ingredient to generalize Kat\v{e}tov's construction to arbitrary
\Fraisse structures is the presence of a
stationary independence relation. For the following
$a,b,\dots$ denote finite tuples, by  
$A,B,C\dots$ we denote small, i.e. finitely generated structures,
whereas $X,Y,\dots$ stand for arbitrary countable ones. Given substructures
$A$ and $B$, the substructure generated by their union is denoted with
$\langle
AB\rangle$. By a type over $X$, we mean a set of $L(X)$ formulas $p(x)$ with
free variables $x$ which is maximal satisfiable in $M$. The type $\tp(a/X)$ of
a tuple $a$ over $X$ is the set of all $L(X)$ formulas which are satisfied by
$a$. 
\begin{defi}[(Local) Stationary Independence Relation]
\label{SIR}
Assume $M$ to be a homogeneous $L$-structure.
A ternary relation $\ind$ on the finitely generated substructures of $M$ is
called a \demph{stationary independence relation} (\demph{SIR}) if the
following
conditions are satisfied:
\begin{enumerate}[label=SIR\arabic*]
 \item\label{invarianz} \upshape{(Invariance)}. The independence of finitely
generated substructures in $M$ only depends on their type. In particular, for 
any
automorphism $f$ of $M$, we have $A\ind_{C}B$ if and only if
$f(A)\ind_{f(C)}f(B)$.
 \item\label{symmetrie}\upshape{(Symmetry)}. If $A\ind_CB$, then $B\ind_CA$.
\item\label{monotonie} \upshape{(Monotonicity)}. If $A\ind_{C}\struc{BD}$, then
$A\ind_{C}B$ and $A\ind_{\struc{BC}}D$.
\item\upshape{(Existence)}. \label{existenz} For any $A,B$ and $C$ in $M$, there
is some
$A'\models \tp(A/C)$ with $A'\ind_{C}B$. 
\item\upshape{(Stationarity)}. \label{statio} If $A$ and $A'$ have the same
type over $C$ and are both independent over $C$ from some set $B$, then they
also have the same type over $\struc{BC}$. 
\end{enumerate}
If the relation $A\ind_{C}B$ is only defined for nonempty $C$, we
call
$\ind$ a \demph{local} stationary independence relation.\hfill$\dashv$
\end{defi}
\begin{rem}
 Any SIR also fulfills the following
property, which was part of the original definition in \cite{Tent2012}:
\begin{enumerate}
\item[SIR6]\label{transitivitat} \upshape{(Transitivity)}. If $A\ind_{C}B$ and
$A\ind_{\struc{BC}}D$, then $A\ind_{C}\struc{BD}$.
\end{enumerate}
 To see that, consider
$A,B,C$ and $D$ with $A\ind_{C}B$ and $A\ind_{\struc{BC}}D$. We have to show 
that this
implies the
independence of $A$ and
$\struc{BD}$ over $C$. By Existence there is some $A'\equiv_C A$ with
$A'\ind_{C}\struc{BD}$.
By Monotonicity and Stationarity we get $A'\equiv_{BC}A$. Again by Monotonicity 
and Stationarity, we obtain $A\equiv_{BCD}A'$ and
hence
$A\ind_{C}\struc{BD}$, as desired. \hfill$\dashv$
\end{rem}

 For a (local) SIR $\ind$ defined on some
homogeneous structure $M$, we call the pair $(M,\ind)$ a \demph{(local)
SI-structure}.
If the interpretation of $\ind$ in $M$ is  clear or irrelevant, we will refer to
$M$ alone as an \SIR.

\begin{rem}
 If $A$ and $A'$ in $M$ have the same quantifier free (qf-)type over some
$B\subseteq M$, then the map $AB\mapsto A'B$ is a partial
isomorphism. In homogeneous structures such
a map extends to
an automorphism of the whole structure $M$, fixing $B$ and sending $A$ to
$A'$. As we will exclusively work inside homogeneous structures for the rest
of the
article, note that $A$ and $A'$ have the same qf-type over
$B$
(denoted by $\tp^{qf}(A/B)=\tp^{qf}(A'/B)$) if and only if there is an
automorphism
of $M$ that fixes $B$ pointwise and maps $A$ to $A'$ (write $A\equiv_B A'$). 
\end{rem}
A necessary condition for a given structure to carry a
stationary
independence relation is given by the following fact (cf. \cite{Tent2012},
Proof of Lemma 5.1).
\begin{fact}
\label{disjunktamalgam}
 Algebraic and definable closure coincide in an \SIR $M$, i.e.
$\acl(X)=\dcl(X)$ for all $X\subset M$. 
\end{fact}

Note that Fact \ref{disjunktamalgam} also holds in local
\SIRs for nonempty $X$. Furthermore, the
equality $\acl(\emptyset)=\dcl(\emptyset)$ is true if and only if either
$\acl(\emptyset)=\emptyset$ or every automorphism has a fixed point.

\hspace{1pt}

This characterization of algebraic closures in \SIRs already implies that we 
cannot define a stationary independence relation on every \Fraisse structure: 
 The class of finite 
fields in a 
fixed characteristic $p$ forms a \Fraisse class. In its limit, the 
algebraically closed field $\bar{\F}_p$ of characteristic $p$, algebraic and 
definable closure differ. Thus Lemma \ref{disjunktamalgam}
states that no stationary independence relation can be defined on
$\bar{\F}_p$. As mentioned before, its automorphism group is the torsionfree 
group $\hat{\Z}$, which does not embed any finite group. Thus, the group 
$\Aut(\bar{\F}_p)$ is not universal for the class of automorphism groups of 
substructures of $\bar{\F}_p$. 

On the other hand, the rationals as a dense linear order form another example 
of a \Fraisse structure which does not allow a notion of stationary 
independence, 
but still has a universal automorphism group. These two examples show that the 
absence of a notion of stationary independence within a \Fraisse structure 
does not decide about the universality of its automorphism group.

Nevertheless, several examples of \Fraisse structures admitting
a stationary independence relation are known. Amongst them are the 
rational Urysohn space and -sphere as well as \Fraisse limits of
rational free amalgamation classes \cite{Bilge2012}.
More examples, also including a non-relational \SIR, will be discussed in
detail in section \ref{sec:ngon}. \\

Unlike forking in simple theories, which is 
uniquely determined by its properties, a \Fraisse structure can carry 
different notions of stationary independence: For an example, consider the 
random graph, and define finite subgraphs $A$ and $B$ to be independent 
over 
some finite subgraph $C$ if and only if $A\cap B\subseteq C$ and every vertex 
in $A\setminus C$ is connected to every vertex in $B\setminus C$. It is not 
hard to verify that this defines a stationary independence relation. On 
the other hand, the class of finite graphs is a free amalgamation class, 
whence another
stationary independence relation is given by the free amalgam of $A$ and $B$ 
over $C$, i.e. $A$ and $B$ are defined to be independent over $C$ if and only if 
$A\cap B\subseteq C$ and no vertex 
in $A\setminus C$ is connected to a vertex in $B\setminus C$.\\

We have to develop some tools to mimic Kat\v{e}tov's construction of 
Urysohn's space. In order to merge certain small extensions over any embeddable 
infinite substructure in an independent way,
we will need to extend the independence notion to arbitrary 
base sets: 
\begin{defi}\label{InfInd}
 Let $M$ be an \SIR. Two substructures $A$ and $B$ are \demph{independent} over 
$X\subseteq M$ (write
$A\ind_{X}B$), if and only if there
is some
finitely generated $C\subset X$ such that $A\ind_{C'}B$ for every finitely
generated $C'\subset X$ containing $C$. 
\end{defi}

Notice that in the examples of \SIRs mentioned above, the independence relation 
is
naturally defined between arbitrary sets and coincides with the one given in the
previous definition. 

\begin{lem}
 The independence relation $\ind$ extended to arbitrary base sets
satisfies all the properties of an SIR except possibly Existence. 
\end{lem}
\begin{proof}
Invariance and Monotonicity easily follow from the
definition. To see
Transitivity, assume $X\subseteq M$ and finitely generated $A,B$ and $D\subset
M$ given
with
\begin{equation*}
 A\ind_{X}B\text{ and }A\ind_{\struc{XB}}D.
\end{equation*}
By definition, there are $C_1$ and $C_2\subset X$ such that $C_1$
and $\struc{C_2B}$ are finite supports for the first and the second
independence respectively, i.e.
$A\ind_{C'}B$ (resp. $A\ind_{C'}D$) for every finitely
generated $C'\subset X$ (resp. $C'\subset \struc{XB}$) containing $C_1$ (resp.
$\struc{C_2B}$).
If we set $C:=\langle C_1C_2\rangle$, then every finitely generated $C'\subset
X$ containing $C$ satisfies 
\begin{equation*}
 A\ind_{C'}B\text{ and }A\ind_{\struc{C'B}}D\text{, hence
}A\ind_{C'}\struc{BD},
\end{equation*}
which yields Transitivity. 

To prove Stationarity, note that two different realizations of the same type
over $X$,
both independent from some finite set $B$, have different types over
$\struc{XB}$
if and only if there is some finite subset $C'$ of $X$ such that their types
differ already over $\struc{C'B}$. 
\end{proof}

The homogeneity of $M$ allows us furthermore to speak of independence
between subsets of 
embeddable structures.
We denote by \demphmath{$\Komega$} the class of all
structures embeddable into $M$, i.e. the class of all structures whose skeleton
is contained in the skeleton of $M$. For some $Y\in \Komega$
with
substructures $A,B$ and $X$, we
say that $A$ is independent from $B$ over $X$ if the same
is true for one, and hence for every embedding of $Y$ into $M$.

\section{A General \kat Construction}

Since the independence relation is defined only in one model and
need not be part of the theory, we cannot ensure the existence of
independent extensions for types over base sets which
are not necessarily finitely generated. Never\-theless, a variant of
Existence for certain types, called finitely supported, can be deduced.

\begin{defi} Let $(M,\ind)$ be an \SIR, and $\langle AX\rangle$ arbitrary 
in $\Komega$. We say that $\struc{AX}$ is \demph{finitely 
supported} (over $X$), if there is a finitely generated subset $C\subseteq X$ 
such that $A\ind_CD$ for all finitely generated $D\subseteq X$. In that case we 
also write $A\ind_CX$ and refer to $C$ as
a \demph{support} of $\struc{AX}$ over $X$. Furthermore, we call a quantifier 
free type $\pi(x)$ over $X$ \demph{finitely supported}, if it defines a 
finitely 
supported $\Komega$-structure. 
\end{defi}

Loosely speaking, a type $\pi(x)$ over $X$ is finitely supported if its
realizations are independent from the base set over some finitely generated
substructure $C$ of $X$. It is not hard to see that every finitely generated
$D\subseteq X$ that contains $C$ is again a
support for $\pi(x)$, so that for any finite family of finitely supported 
types over
$X$ we can choose a common support. 

Let us denote by \demphmath{$S^{qf}(X)$} the set of all quantifier-free
types
over $X$. In the
following, we show some useful properties of finitely supported
types and structures.

\begin{lem} Assume $X$ to be a $\Komega$-structure. 
 \begin{itemize}
 \item [i)] \label{extension} Suppose $C\subseteq X$ is finitely generated and
$\pi:=\pi(x)$ a
qf-type over $C$ realized in $M$. Then $\pi$ has a unique
extension
$\tilde{\pi}\in S^{qf}(X)$ which is finitely supported over $X$ with support 
$C$.
\item [ii)] Let $\struc{AX}$ and $\struc{BX}\in \Komega$ be finitely
supported
over
$X$. Then
there is some $\struc{A'B'X'}\in \Komega$ with 

$$\struc{A'X'}\cong \struc{AX},\; \struc{B'X'}\cong \struc{BX}
\text{ and } A'\ind_{X'}B'.$$ 

Furthermore, the structure $\struc{A'B'X'}$ is again finitely
supported over
$X$. 
\end{itemize}
\end{lem}

\begin{proof} By choosing an arbitrary embedding, we may assume $X$ to
be a substructure of $M$.
\begin{itemize}
\item [i)] Write $X$ as the limit of 
a chain $X=\bigcup_{n\in\omega}
C_n$ with $C_0:=C$. Inductively we can construct a chain of types
$\pi_0\subseteq \pi_1\subseteq \dots$ by setting $\pi_0:=\pi$ and for
each $n>0$, we set $\pi_n:=\tp^{qf}(A_n/C_n)$, where $A_n\subseteq M$ 
with $A_n\models \pi_{n-1}$ and
$A_n\ind_{C_{n-1}}C_n$. By
compactness and Transitivity, the set
$\tilde{\pi}:=\bigcup_{n\in\omega}\pi_n$ is a finitely supported type over $X$
with
support $C$. Note that $\tilde{\pi}$ defines again a $\Komega$-structure 
$\struc{AX}$, as every finite subset of $\struc{AX}$ is embeddable in some 
$\struc{A_nC_n}\subseteq M$. The uniqueness of $\tilde{\pi}$ now follows from 
Stationarity.
\item [ii)]  Let $C\subset X$ be some common support
of $\struc{AX}$ and $\struc{BX}$ over $X$. By
Existence we find realizations $A_1$ (resp. $B_1$) of
the qf-type of $A$ (resp. $B$) over $C$ in $M$ such that $A_1\ind_{C}B_1$. Part
i) allows us to extend the type $\tp^{qf}(A_1B_1/C)$ to some finitely supported
type
$\pi$ over $X$ with support $C$, which defines a
$\Komega$-structure $\struc{A_2B_2X}$. As $A_2$ (resp. $B_2$) and $A$ (resp.
$B$) have the same qf-type over $C$ and are both independent from $X$ over $C$,
Stationarity implies that $\struc{AX}\cong \struc{A_2X}$ (resp.
$\struc{BX}\cong \struc{B_2X}$). Furthermore, for any finitely generated
$C'\subset X$ containing $C$ we have
\begin{equation*}
 \struc{A_2B_2}\ind_{C}C'\text{ and }A_2\ind_{C}B_2.
\end{equation*}
So $A_2\ind_{C'}B_2$ by Monotonicity and Transitivity, and thus
$A_2\ind_{X}B_2$. That finishes the proof.
\end{itemize}
\end{proof}

The second part of the above lemma shows how to independently glue certain
structures over arbitrary $\Komega$-base sets. As we will see below, the
$\Komega$-structure described in Lemma \ref{extension}.ii)
is unique up to isomorphism. This justifies the following definition.

\begin{defi}
 Assume  $\struc{AX}$ and
$\struc{BX}$ to be finitely supported $\Komega$-structures. The
structure $\struc{A'B'X'}\in \Komega$ obtained in Lemma \ref{extension}.ii) is
called the \demph{SI-amalgam} of $\struc{AX}$ and $\struc{BX}$ over $X$ and
denoted by \demphmath{$A\amalgamo{X}B$}.
\end{defi}

Since an SI-amalgam is again a finitely supported
$\Komega$-structure, we
can amalgamate finite families of finitely supported
$\Komega$-structures. This process behaves well under
permutations of the given finite family.

\begin{lem}\label{lem39}
 The SI-amalgam of two structures is unique up to isomorphism. Moreover,
SI-amalgamation is 
commutative and associative, meaning that for
given finitely supported $\struc{AX},\struc{BX}$ and $\struc{CX}$, the
SI-amalgams $A\amalgamo{X}B$ and $B\amalgamo{X}A$
(resp. \linebreak$(A\amalgamo{X}B)\amalgamo{X}C$ and
$A\amalgamo{X}(B\amalgamo{X}C)$)
 are isomorphic. 
\end{lem}
\begin{proof}
Let two $\Komega$-structures $\struc{A_1B_1X_1}$ and $\struc{A_2B_2X_2}$ be
given with $A_iX_i\cong AX$ and $B_iX_i\cong BX$ as well as $A_i\ind_{X_i}B_i$. 
By Lemma
\ref{extension}.ii), the structures $\struc{A_iB_iX_i}$ are again finitely
supported with support $C_i\subseteq X_i$. Because $A_iX_i\cong AX$, we may 
assume that $A_iC_i\cong AC$ for some $C\subset X$. Note that
we can pick the $C_i$'s such that they also witness the independence of $A_i$ 
and $B_i$ over $X_i$, i.e.
$A_i\ind_{C'}B_i$ for all $C'\subset X_i$ with $C_i\subseteq C'$. By 
Stationarity and
Invariance, it suffices to show that $A_1B_1C_1\cong A_2B_2C_2$. Let us assume
that the structures are embedded into $M$. By homogeneity, the partial
isomorphism \texttt{$f:A_1C_1\rightarrow A_2C_2$} extends to
$\struc{A_1B_1C_1}$, yielding a copy $f(B_1):=B_2'$ of $B_1$. The structures 
$B_2$ and
$B_2'$ have the same type over $C_2$ and are both independent from $A_2$ over
it. Hence, there is an automorphism $g\in\Aut(M)$ that fixes $A_2C_2$ and sends
$B_2'$ to $B_2$. Finally, the map $g\circ f:\struc{A_1B_1C_1}\rightarrow
\struc{A_2B_2C_2}$ provides the desired isomorphism.

Commutativity follows directly from Symmetry of our
independence relation.
It remains to show that the amalgamation process is associative. To see that,
assume $\struc{A_1B_1C_1X_1}=A\amalgamo{X}(B\amalgamo{X}C)$. By definition of
the SI-amalgam, we have $A_1\ind_{X_1}B_1C_1$ and $B_1\ind_{X_1}C_1$, whence 
\begin{equation*}
 (1)\; A_1\ind_{X_1}B_1\text{ and }(2)\; A_1B_1\ind_{X_1}C_1,
\end{equation*}
by Monotonicity and Transitivity. Now $(1)$ yields
$\struc{A_1B_1X_1}=A\amalgamo{X}B$, whereas $(2)$ concludes that
\begin{equation*}
 \struc{A_1B_1C_1X_1}=(A\amalgamo{X}B)\amalgamo{X}C.
\end{equation*}
This implies $A\amalgamo{X}(B\amalgamo{X}C)\cong(A\amalgamo{X}B)\amalgamo{X}C$, 
as
desired.
\end{proof}

Lemma \ref{lem39} guarantees that the order in which we amalgamate a finite 
family
of structures is irrelevant. Hence, for a given finite family of
finitely supported $\Komega$-structures
$\{\struc{A_iX}, i\in n\}$, it makes sense to write: 
\begin{equation*}
 \underset{i\in n}{\amalgamo{X}}A_i:=((\dots
((A_0\amalgamo{X}A_1)\amalgamo{X}A_2)\dots) \amalgamo{X}A_{n-1}).
\end{equation*}

This amalgamation will be the main tool for developing a general analogue
of the
so-called \kat spaces in the non-metric setting. When we now move on to
countable families
$\mathcal{F}:=\{\struc{A_iX},i\in\omega\}$ of finitely supported structures over
$X$, note that every
finite amalgam $\underset{i\leq n}{\amalgamo{X}}A_i$ can naturally be embedded
into $\underset{i\leq n+1}{\amalgamo{X}}A_i$, so that the family
$(\underset{i\leq
n}{\amalgamo{X}}A_i)_{i\in\omega}$ is a directed system. The structure
 \demphmath{$\underset{i\in\omega}{\amalgamo{X}}A_i$} generated by the limit of
this system is still a $\Komega$-structure, called the \demph{SI-amalgam}
of $\{\struc{A_iX},i\in\omega\}$. As we are mainly interested in extensions of
automorphisms, the following lemma will be
useful further on.
\begin{lem} \label{lemamalgamationofisos}
 Let $\{\struc{A_iX},i\in\omega\}$ be a countable family of
$\Komega$-structures, finitely supported over $X$. Let furthermore 
$\{f_i:A_iX\rightarrow A_{\sigma(i)}X, i\in\omega\}$ be a family of
isomorphisms, where $\sigma$ is a
permutation of $\omega$ and ${f_i}|_{X}={f_j}|_{X}$ for all $i,j$. Then the
union $\bigcup_{i\in\omega}f_i$ induces an
automorphism of the SI-amalgam $\underset{i\in\omega}{\amalgamo{X}}A_i$.
\end{lem}
\begin{proof}
 We will establish the statement for the SI-amalgam of two structures. The claim
follows via induction. Hereby, surjectivity in the limit process is given by
the surjectivity of $\sigma$. As all of the structures are in
$\Komega$, we
may take $A_0\amalgamo{X}A_1$ and $A_{\sigma(0)}\amalgamo{X}A_{\sigma(1)}$ to be
substructures of $M$ and hence the $f_i$'s to be partial isomorphisms. It
suffices to show that $g:=f_1\cup f_2$ restricted to any finitely generated
substructure of $A_0\amalgamo{X}A_1$ is again a partial isomorphism. 

Choose $D\subset X$ arbitrary. As $g_{|A_0D}={f_0}_{|A_0D}$, the restriction
of
$g$ to $A_0D$ defines a partial isomorphism between finitely generated
substructures of $M$ and hence it extends to an automorphism
$\tilde{g}\in\Aut(M)$. Denote by $B_{\sigma(1)}$ the image of $A_1$ under
$\tilde{g}$. Then $A_{\sigma(1)}$ and $ B_{\sigma(1)}$ have the same type over
$D$, as $\tilde{g}_{|D}={f_0}_{|D}={f_1}_{|D}$, and are both independent from
$A_{\sigma(0)}$
over $D$ by Invariance.
Stationarity implies now that they even have the same type over
$DA_{\sigma(0)}$, whence there is some $h\in\Aut(M)$ fixing  $DA_{\sigma(0)}$
and sending $ B_{\sigma(1)}$ to $A_{\sigma(1)}$. Thus
$h\tilde{g}_{|A_0A_1D}=g_{|A_0A_1D}$ is a partial isomorphism and we are done.
\end{proof}

The above lemma yields a crucial property of SI-amalgams
needed to imitate Kat\v{e}tov's construction. We can now define a chain of
\kat
spaces in this setting. 

 Given an arbitrary $\Komega$-structure $X$, denote by
$S^{\fin}(X)$ the space of all finitely supported qf-types
over $X$. The set $S^{\fin}(X)$ gives
rise to a countable family $\mathcal{F}$ of finitely supported structures
and, after
fixing an arbitrary enumeration, we
can build the SI-amalgam of that family. 

\begin{defi}
 After choosing an arbitrary enumeration $S^{\fin}(X)=\{\struc{A_iX}\mid
i\in\omega\}$ of the space of qf-types over some
$\Komega$-structure $X$, 
let \demphmath{$E_1(X)$}$:=\underset{i\in\omega}{\amalgamo{X}}A_i$  be
the SI-amalgam of that family and call it the
\demph{first \kat space} of $X$.
\end{defi}
One can show that $E_1(X)$ does not depend on the chosen enumeration of
$S^{\fin}(X)$.
Moreover, the space $E_1(X)$ is again a
$\Komega$-structure, whence we may iterate the procedure and
thereby construct inductively the $n$-th \kat spaces $E_n(X)$ of $X$ as follows:
\begin{eqnarray*}
E_0(X)&:=&X,\\
 E_{n+1}(X)&:=&E_1(E_n(X)).
\end{eqnarray*}
This family of \kat spaces comes equipped with natural embeddings between its
members and therefore it forms an inductive system. The limit \demphmath{$E(X)$}
of that system will be called the \demph{\kat limit} of $X$. 

In order to prove Theorem \ref{mainthm}, it remains to show that the \kat limit
of an arbitrary
$\Komega(M)$-structure $X$ is isomorphic to $M$ and $\Aut(X)$ embeds 
continuously into
$\Aut(E_{1}(X))$. 
\begin{lem}\label{isomorphic}
 Assume $M$ to be an \SIR and consider an arbitrary $X\in\Komega(M)$. The \kat 
limit $E(X)$ is isomorphic to $M$.
\end{lem}
\begin{proof}
 Let $\mathcal{K}$ be the skeleton of $M$. As $M$ is homogeneous,
the class $\mathcal{K}$ is a \Fraisse class and $M=\Fr(\mathcal{K})$ is
$\mathcal{K}$-saturated. Any two
countable $\mathcal{K}$-saturated structures are isomorphic, whence it suffices
to prove $\mathcal{K}$-saturation for $E(X)$ to establish the lemma.

We will first show that $\mathcal{K}$ is exactly the skeleton of $E(X)$. It is
easy to see that $\mathcal{K}$ is contained in the skeleton, as for any finitely
generated structure $A=\langle a \rangle$ of $M$, the type
$\tp^{qf}(a/\emptyset)$
determines completely $A$ and it can
be extended to a finitely supported type over $X$ by Lemma \ref{extension}.i).
Hence,
there is a copy of each structure from $\mathcal{K}$ inside the
first \kat space
$E_1(X)$, and thus also in the limit $E(X)$.

For the other direction, let $A\subset E(X)$ be an arbitrary finitely
generated substructure. Then there is some $n\in\omega$ such that
$A\subset E_n(X)$. Since all the \kat-spaces $E_n(X)$ are
$\Komega$-structures, it follows that $A\in\mathcal{K}$ by definition, so
$\mathcal{K}=\mathcal{K}(E(X))$ as
desired.

It remains to show that $E(X)$ is rich with respect to
$\mathcal{K}$. Consider some finitely generated substructure
$A\subset E(X)$ and a $\mathcal{K}$-structure $B$ with $f:A\hookrightarrow
B=\langle Ab\rangle$.
Again, the structure $A$ is contained in some $E_n(X)$ and we can extend
$\tp^{qf}(b/A)$ to a finitely supported type over $E_n(X)$. Since a
realization of this type occurs in $E_{n+1}(X)$, we can embed $B$ in $E(X)$
over $A$. 

Consequently, the countable \kat-limit of an arbitrary $\Komega$-structure $X$
is $\mathcal{K}$-saturated and hence isomorphic to $M$. 
\end{proof}
\begin{lem}\label{extension2}
 Let $M$ be an \SIR and $X\in\Komega$ arbitrary. Then
$\Aut(X)$ embeds continuously into $Aut(E_{1}(X))$. In particular, the group 
$\Aut(X)$ embeds continuously into 
$\Aut(E(X))$.
\end{lem}
\begin{proof}
Consider an
automorphism $f\in \Aut(X)$ and observe that $f$ induces a
permutation $\sigma_f$ of $S^{\fin}(X)=(\struc{A_iX}\mid i\in\omega)$, the space
of all finitely supported qf-types
over $X$.
By Lemma
\ref{lemamalgamationofisos}, this gives rise to an automorphism of the amalgam
$\amalgamo{i\in \omega}A_i=E_1(X)$. Hence, for every $f\in\Aut(X)$, there is an
automorphism
$\hat{f}\in \Aut(E_1(X))$ which extends $f$. As for every finitely
supported type $\pi(x)$ over $X$
there exists a unique $k\in\omega$ with $A_k\models \pi(x)$, these extensions
behave
well under multiplication, i.e. $\sigma_g\circ (\sigma_f)^{-1}=\sigma_{g\circ
(f^{-1})}$. Therefore, the set $\{\hat{f}\mid f\in\Aut(X)\}$
forms a subgroup of $\Aut(E_1(X))$. Denote by $\iota$ the map that sends 
$f\in\Aut(X)$ to $\hat{f}\in\Aut(E_1(X))$. If we identify
the structures
coming from $S^{\fin}(X)$ in $E_1(X)$ again with $\{\langle A_iX\rangle\mid
i\in\omega\}$, the
isomorphic copy of $\Aut(X)$ inside $\Aut(E_1(X))$ consists of the
subgroup of all $f\in\Aut(E_1(X))$ such that $f|_{X}$ is
in $\Aut(X)$ and $f$ induces a
permutation on $\{A_i\mid i\in\omega\}$. 

It remains to show that $\iota:\Aut(X)\rightarrow 
\Aut(E_1(X))$ is a continuous embedding: Let $\hat{f}\in\im(\iota)$ 
be an automorphism of $E_1(X)$. For an 
arbitrary finite subset $a\subseteq E_1(X)$, let $u:=\hat{f}|_a: 
a\rightarrow E_1(X)$ be the restriction of $\hat{f}$ to $a$ and 
$\mathcal{O}_u:=\{g\in\Aut(E_1(X))\mid g|_{a}=u\}$ the basic open set defined 
by $u$ containing $\hat{f}$. We have to show that the preimage of 
$\mathcal{O}_u$ under $\iota$ contains again an open subset.
As $a$ is finite, we can 
choose $A_{i_1},\dots, A_{i_n}$ from above and $C_0\subseteq X$ such that $a$ 
is definable over $C_0\cup\bigcup_{j=1,\dots,n}A_{i_j}$. 
Since the $A_{i_j}$ correspond to finitely supported extensions of $X$, for 
each 
$j=1,\dots, n$ there exists a $C_j\subseteq X$ with $A_{i_j}\ind_{C_j}X$. Set 
by $C:=\bigcup_{i\leq n}C_i$ and $v:=\hat{f}|_C$ the restriction of $\hat{f}$ 
to $C$.
We claim that $\mathcal{O}_v:=\{g\in\Aut(X)\mid g|_{C}=v\} \subseteq 
\Aut(X)$ is contained in the preimage of $\mathcal{O}_u$ under $\iota$. Let 
$g\in\mathcal{O}_v$ be an arbitrary automorphism of $X$ that extends $v$ and 
$\hat{g}$ its extension to $E_1(X)$. As by assumption 
$A_{i_j}\ind_CX$ and $\hat{g}(C)=\hat{f}(C)$, Invariance implies
\begin{equation}
 \hat{g}(A_{i_j})\ind_{\hat{f}(C)}X\text{ and 
}\hat{f}(A_{i_j})\ind_{\hat{f}(C)}X.
\end{equation}
Thus, Stationarity yields $\hat{g}(A_{i_j})\equiv_X\hat{f}(A_{i_j})$. As both 
$\hat{f}$ and $\hat{g}$ are in the image of $\iota$, the images of $A_{i_j}$ 
under each of the two maps is again one of the $A_k$. On the other hand, every 
finitely supported extension of $X$ has only been realized once within the 
$A_k$'s, whence $\hat{g}(A_{i_j})=\hat{f}(A_{i_j})$ for all $j=1,\dots, n$. In 
particular, we get $\hat{g}(a)=\hat{f}(a)$ and $\hat{g}\in\mathcal{O}_u$. This 
proves 
that the embedding $\iota:\Aut(X)\rightarrow \Aut(E_1(X))$ is continuous.
\end{proof}

With Lemmas \ref{isomorphic} and \ref{extension2} at hand, the main theorem
now follows easily.
\begin{thm}[Main Theorem] \label{mainthm}
Let $M$ be a countable \SIR and $\Komega$ be the class of all
countable structures
embeddable into
$M$. Then the automorphism group $\Aut(M)$ is universal for the class
$\Aut(\Komega):=\{\Aut(X)\mid X\in\Komega\}$, i.e. every group in
$\Aut(\Komega)$ can be continuously embedded as a subgroup into $\Aut(M)$.  
\end{thm}
\begin{proof}
 Assume $X$ to be a $\Komega$-structure. Lemma \ref{extension2} yields that the
automorphism group of $X$ can be continuously embedded as a subgroup into the 
automorphism
group of its \kat limit $E(X)$. As this limit is isomorphic to $M$ by Lemma
\ref{isomorphic}, we conclude that $\Aut(X)$ can also be continuously
embedded as a
subgroup into $\Aut(M)$.
\end{proof}

\section{Generalized n-gons}\label{sec:ngon}

Several examples of \Fraisse structures with (local)
stationary independence were already mentioned in the introduction, amongst them
pure structures, relational \Fraisse
limits with free
amalgamation and the rational Urysohn space and sphere. In all these cases
the construction given above yields that their automorphism groups are universal
with respect to the class of automorphism groups of their countable
substructures.

Another \SIR in a relational language, yet without free amalgamation,
is the countable universal partial order $\mathcal{P}$. Given finite
partial orders $A$ and $B$ with a common substructure $C$, we define the amalgam
$A\amalgam_CB$ of $A$
and $B$ over $C$ as the structure consisting of the disjoint union
of $A$ and $B$
over $C$ such that $a\leq b$ (resp. $b\leq a$) if and only if there is some 
$c\in C$ with
$a\leq c\leq b$ (resp. $b\leq c\leq a$). It is easy to check that the relation
$A\ind_{C}B$ defined
by $\struc{ABC}\cong \samalgamo{A}{C}{B}$ provides a stationary independence
relation on $\mathcal{P}$. \\

Further, non relational examples have recently been provided
by Baudisch \cite{Baudisch2014}, who shows that graded Lie algebras
over finite fields and $c$-nilpotent groups of exponent $p$ with an extra
predicate for a central Lazard series are \SIRs, whence
their automorphism group is universal. \\

In the following, we will exhibit yet another homogeneous non
relational structure that admits a
stationary independence relation: The countable universal generalized $n$-gon
$\Gamma_n$, which arises
as the \Fraisse limit of certain bipartite graphs \cite{Tent2011}.

A graph $G=(V_G,E_G)$ is \demph{bipartite} if we can partition its 
vertex set $V_G=V_1\dot\cup V_2$ in such a way that every vertex in $V_1$ is 
only 
connected to vertices in $V_2$ and vice versa. We equip graphs with the 
\demph{graph metric} $d^G$, where $d^G(x,y)$ is the length of the shortest path 
from $x$ to $y$ in $G$. Furthermore, the \demph{diameter} of $G$ is the 
smallest number 
$n\in\N$ such that the distance between every two vertices $x$ and $y$ is at 
most $n$. If no such number exists, we say that $G$ is of infinite diameter. 
Finally, the \demph{girth} of $G$ is the length of a shortest cycle in $G$. By 
a \demph{subgraph} $H\subseteq G$ we mean an induced subgraph.

\begin{defi}
 A \demph{generalized} $n$\demph{-gon} $\Gamma$ 
is a bipartite graph of
diameter $n$ and girth $2n$. 
\end{defi}

Generalized $n$-gons were introduced by Jaques Tits, who developed the 
theory of buildings. Note that an example of a generalized $3$-gon is given by 
a projective plane. The class of generalized $n$-gons coincides with the class 
of spherical buildings of rank $2$.

We consider generalized $n$-gons $\Gamma$ in the
language $L_n=\struc{P,f_k\mid k=0,\dots,n}$, where $P$ is a
predicate for the sort of the partition of the vertex set and the $f_k$'s are 
binary
functions with $f_k(x,y):=x_k$ if
$d^{\Gamma}(x,y)=l\geq k$ and there is a unique shortest path
$p=(x=x_0,\dots,x_k,\dots,x_l=y)$ from $x$ to $y$. 
Such a unique path always exists, if $l<n$, as otherwise there 
would be a non-trivial cycle of length $2l<2n$, 
contradicting the assumption on the girth of $\Gamma$. If there is no such 
unique path or $d^{\Gamma}(x,y)<k$, we set $f_k(x,y):=x$. 
Note that the edge
relation is definable within this language as two vertices $x$ and $y$ are
incident if and only
if $x\neq y$ and $f_1(x,y)=y$. Furthermore, if $\Delta\subseteq \Gamma$ is 
a generalized $n$-gon contained in $\Gamma$, then $\Delta$
is generated by a subgraph $A\subseteq \Delta$ as an $L_n$-structure, if and 
only if $\Delta$ is the smallest generalized $n$-gon in $\Gamma$ containing 
$A$. 

Given any connected bipartite graph $G$ without cycles of length less than $2n$,
we can build an $L_n$-structure inductively as follows: Set 
$\mathcal{F}_0(G):=G$.
Assume the graph $\mathcal{F}_i(G)$ has already been constructed. For any
pair $(x,y)$ in $\mathcal{F}_i(G)$ with distance $n+1$ in $\mathcal{F}_i(G)$, 
we add a \demph{new path} $p_{x,y}=(x=x_0,x_1,\dots,x_{n-2},x_{n-1}=y)$ 
from $x$ to $y$ of length $n-1$, i.e. a path from $x$ to $y$ of length $n-1$ 
such that all the $x_i$ for $i=1,\dots, n-2$ are new vertices. 
Clearly, the graph $\mathcal{F}_i(G)\cup p_{x,y}$ still does not have cycles of 
length less than $2n$,  as every such cycle would have to contain the whole 
path $p_{x,y}$, but if we could complete $p_{x,y}$ to a cycle of length less 
than $2n$, there existed a path of length at most $n$ between $x$ and $y$ 
in 
$\mathcal{F}_i(G)$, contradicting $d^{\mathcal{F}_i(G)}(x,y)=n+1$. Now we define
\begin{equation*}
 \mathcal{F}_{i+1}(G):=\mathcal{F}_i(G)\cup\{p_{x,y}\mid x,y\in 
\mathcal{F}_i(G) \text{ and }d^{\mathcal{F}_i(G)}(x,y)=n+1\}
\end{equation*}
and call the graph
$\mathcal{F}(G):=\bigcup_{i\in\omega}\mathcal{F}_i(G)$ the \demph{free
\demphmath{$n$}-completion} of $G$. Observe that
$\mathcal{F}(G)$ is a generalized $n$-gon.  

We now consider the class $\mathcal{C}_n$ of all free
$n$-completions of finite connected bipartite graphs without cycles of length 
less than $2n$ and their $L_n$-substructures. 

Tent shows in \cite{Tent2011}
that $\mathcal{C}_n$ is a \Fraisse class and hence it admits a countable
homogeneous limit $\Gamma_n$. She also provides a
characterization to recognize free $n$-completions of finite graphs via some
corresponding weighted Euler characteristic. We will lift that characterization 
to infinite
graphs and use it to define an independence relation on $\Gamma_n$. 

\begin{defi}
  Consider the function $\chi_n$ on finite graphs $H$ with vertex set
$V_H$ and edge set $E_H$ defined by $\chi_n(H):=(n-1)|V_H|-(n-2)|E_H|$. For 
arbitrary, possibly infinite graphs $X\subseteq Y$, we say that $X$ is 
\demph{n-strong} in $Y$
(write $X\leq_n Y$) if and only if for all
finite $H\subseteq Y$ we have
$\chi_n(H/H\cap X ):=\chi_n(H)-\chi_n(H\cap X)\geq 0$.
\end{defi}

\begin{lem}\label{infinitedelta}
Assume $X\subseteq Y$ to be graphs. If $Y$ arises from $X$ by successively 
patching new paths of
length $n-1$, then the following hold:
\begin{itemize}
 \item [i)] $ X\leq_n Y$ and
 \item [ii)] if $X\subseteq Y\subseteq\Delta$ for some graph $\Delta$ with 
$X\leq_n\Delta$, then
$Y\leq_n\Delta$.
\end{itemize}
\end{lem}

\begin{proof}
It suffices to consider the case where $Y$ arises from $X$ by adding one new
path $p_{x,y}=(x=x_0,x_1,\dots,x_{n-2},x_{n-1}=y)$ of length $n-1$ to two 
vertices $x$ and $y$ of $X$, i.e. $Y:=X\cup p_{x,y}$ and $x_i\not\in X$ for 
$i=1,\dots, n-2$.

$i)$ We have to show that 
\begin{equation}\label{chifunction}
 \chi_n(H/H\cap X)=(n-1)(|V_{H}|-|V_{H\cap X}|)-(n-2)(|E_{H}|-|E_{H\cap 
X}|)\geq 0
\end{equation}
 for any finite subgraph $H\subseteq Y$. Clearly, the inequality 
(\ref{chifunction}) holds, if $|V_{H}|-|V_{H\cap X}|\geq |E_{H}|-|E_{H\cap 
X}|$, i.e. if there are more vertices in $H$ outside of $X$ than there are 
edges in $H$ outside of $X$. Thus, as $Y=X\cup p_{x,y}$, the inequality 
(\ref{chifunction}) could only fail, if $H$ contained the entire path 
$p_{x,y}$. But then, 
\begin{eqnarray*}
 \chi_n(H/H\cap X)&=&(n-1)(|V_{H}|-|V_{H\cap X}|)-(n-2)(|E_{H}|-|E_{H\cap 
X}|)\\
&=&(n-1)(|V_{p_{x,y}}|-|\{x,y\}|)-(n-2)|E_{p_{x,y}}|\\
&=&(n-1)(n-2)-(n-2)(n-1)\\
&=&0.
\end{eqnarray*}
Hence, we get $ X\leq_n Y$, as desired.

$ii)$ Consider $H\subseteq \Delta$ finite. We have to show that
$\chi_n(H/H\cap Y)\geq 0$. Let $H'$ be the smallest subgraph of $\Delta$ 
that contains $H$ and the path $p_{x,y}$. As $Y=X\cup p_{x,y}$, all vertices in 
$H'\setminus H$ are also in $H'\cap Y$. This yields 
\begin{eqnarray*}
 \chi_n(H/H\cap Y)&=&(n-1)(|V_H|-|V_{H\cap Y}|)-(n-2)(|E_H|-|E_{H\cap 
Y}|)\\
&=& (n-1)(|V_{H'}|-|V_{H'\cap Y}|)-(n-2)(|E_{H}|-|E_{H\cap 
Y}|)\\
&\geq& (n-1)(|V_{H'}|-|V_{H'\cap Y}|)-(n-2)(|E_{H'}|-|E_{H'\cap 
Y}|)\\
&=&\chi_n(H'/H'\cap Y).
\end{eqnarray*}

One calculates as above that $\chi_n(H'\cap Y)=\chi_n(H'\cap
X)$, because $H'\cap Y= (H'\cap X)\cup p_{x,y}$. This yields
\begin{equation*}
 \chi_n(H/H\cap Y)\geq \chi_n(H'/H'\cap Y)=\chi_n(H'/H'\cap X)\geq 0.
\end{equation*}

\end{proof}

With Lemma \ref{infinitedelta} at hand, we can now prove the following 
characterization of free $n$-completions. The proof follows the proof of 
Proposition 2.5 from \cite{Tent2011}, we included for completeness. 

\begin{lem}\label{equivalentdefis}
For any $L$-structure
$\Delta\in\mathcal{C}_n$
generated by a subset $X\subset \Delta$, the following are equivalent:
\begin{itemize}
 \item [i)] $X$ is $n$-strong in $\Delta$;
 \item [ii)] $\Delta$ is the free $n$-completion of $X$.
\end{itemize}
\end{lem}

\begin{proof}
 \underline{$\Rightarrow$:} Assume $X$ is $n$-strong in $\Delta$. As above, we 
will denote by $\mathcal{F}_k(X)$ the $k$-th step of the free $n$-completion of 
$X$. We show that $\mathcal{F}_k(X)\leq_n \Delta$ for all $k$, which implies 
$\mathcal{F}(X)=\Delta$, as $X$ generates $\Delta$. 

For $k=0$, the statement is given by the assumption, as $\mathcal{F}_0(X)=X$. 
Now assume that 
$\mathcal{F}_k(X)\leq_n \Delta$ and consider $x,y\in \mathcal{F}_k(X)$ 
arbitrary with $d^{\mathcal{F}_k(X)}(x,y)=n+1$. As $\Delta$ is in 
$\mathcal{C}_n$, the conditions on diameter and girth imply that there is a 
unique 
path $p_{x,y}=(x,x_1,\dots,x_{n-2},y)$ between $x$ and $y$ of length $n-1$ in 
$\Delta$. We claim that $x_i\not\in \mathcal{F}_k(X)$ for $i=1,\dots, n-2$, i.e. 
that $p_{x,y}$ is a new path. Choose $1\leq j_1+1<j_2\leq n-1$ such that 
$x_{j_1}, x_{j_2}\in\mathcal{F}_k(X)$ and $x_l\not\in  \mathcal{F}_k(X)$ for all 
$l$ with $j_1< l <j_2$. Let $H:=\{x_{j_1},x_{j_1+1},\dots, x_{j_2}\}$ be the 
path between $x_{j_1}$ and $x_{j_2}$. Because $\mathcal{F}_k(X)$ is 
$n$-strong in $\Delta$, we know that $\chi_n(H/H\cap\mathcal{F}_k(X))\geq 
0$. An easy calculation shows that $\chi_n(H/H\cap\mathcal{F}_k(X))\geq 0$ 
if and only if the path $H$ has at least length $n-1$, i.e. $j_1=0$ and 
$j_2=n-1$. This proves that $p_{x,y}$ indeed is a new path and thus 
$\mathcal{F}_{k+1}(X)\subseteq \Delta$. By Lemma \ref{infinitedelta}.ii), 
$\mathcal{F}_{k+1}(X)$ is even $n$-strong in $\Delta$, which concludes the 
induction.

\underline{$\Leftarrow$:} If $\Delta$ is the free $n$-completion of $X$, it 
arises from $X$ by successively patching new paths of length $n-1$. Thus, Lemma 
\ref{infinitedelta}.i) immediately implies that $X\leq_n \Delta$.
\end{proof}

We can now define a
stationary independence relation on $\Gamma_n$ which corresponds to free
amalgamation. For given graphs $A,B$ and $C$ with
$C\subseteq A,B$, we denote the free amalgam of $A$ and $B$ over $C$ by
$A\amalgam_{C}B$. 

\begin{defi}[Independence Relation on $\Gamma_n$] Let $\Gamma_n$ be the
countable universal homogeneous generalized $n$-gon as introduced above.
For finitely generated substructures $A,B$ and $C\subset \Gamma_n$, we
define $A$ and
$B$ to be
independent over $C$ (write $A\ind_{C}B$) if and only if the free amalgam
$\langle AC\rangle\amalgam_{C}\langle BC\rangle$ is $n$-strong in the
substructure
$\langle ABC\rangle \subseteq \Gamma_n$ generated by $A,B$ and $C$ in
$\Gamma_n$.
\label{ngonindependence}
\end{defi}

Lemma \ref{equivalentdefis} yields that $A\ind_{C}B$ if and only if the
substructure generated
by $ABC$ is exactly the free $n$-completion of $\langle
AC\rangle\amalgam_{C}\langle BC\rangle$. We will use both
characterizations to prove the following main lemma.

\begin{lem}\label{ngonSIR}
 The relation $\ind$ as stated in Definition \ref{ngonindependence} defines
a stationary independence relation on $\Gamma_n$.
\end{lem}

\begin{proof}
 \textit{Invariance} and \textit{Symmetry} are immediate and \textit{Existence}
follows since $\Gamma_n$ is rich with respect
to $\mathcal{C}_n$-structures.

\textit{Monotonicity:} If $A\ind_{C}\struc{BD}$,
we know that
\begin{equation}\label{eq:freeamalg}
 \struc{AC}\amalgam_{C}\struc{BCD}\leq_n \struc{ABCD}.
\end{equation}

We will first prove that this yields the equality
\begin{equation}\label{eq:intersec}
 \struc{ABC}\cap\struc{BCD}=\struc{BC}.
\end{equation}

Recall that for any subgraph $\Gamma'\subseteq \langle ABCD\rangle$ and points 
$x$ and $y$ in $\Gamma'$ with distance $d^{\Gamma'}(x,y)=n+1$, there is a 
unique path 
$p_{x,y}=(x=x_0,x_1,\dots,x_{n-1}=y)\subseteq \langle ABCD \rangle$ from $x$ to 
$y$ of length $n-1$, by the conditions on diameter and girth in generalized
$n$-gons. As before, we will denote by 
$\mathcal{F}_k:=\mathcal{F}_k(\struc{AC}\amalgam_C\struc{BCD})$ the $k$-th step 
of the free $n$-completion of $\struc{AC}\amalgam_C\struc{BCD}$. Inductively 
we define
\begin{eqnarray*}
 \Gamma_0&:=&\struc{AC}\amalgam_C\struc{BC}\text{ and }\\
 \Gamma_{k+1}&:=&\Gamma_k\cup\{p_{x,y}\mid x,y\in \Gamma_k, 
d^{\mathcal{F}_k}(x,y)=n+1\} .
\end{eqnarray*}
Note that $\Gamma_k\subseteq \struc{ABC}$ for each $k$ and 
$\bigcup_{k\in\omega}\Gamma_k$ is a generalized $n$-gon, whence
$\bigcup_{k\in\omega}\Gamma_k=\struc{ABC}$.
Thus, in order to prove (\ref{eq:intersec}), it suffices to show that 
$\Gamma_k\cap \struc{BCD}=\struc{BC}$ for all $k$. We furthermore need 
to prove $\Gamma_k\subseteq \mathcal{F}_k$, so that the calculation of the 
distance $d^{\mathcal{F}_k}(x,y)$ is well defined for arbitrary $x$ and $y$ in 
$\Gamma_k$.

Both statements are clear for $k=0$ by the inequality in (\ref{eq:freeamalg}). 
Now, 
assume that 
$\Gamma_k\subseteq \mathcal{F}_k$ and that 
$\Gamma_k\cap\struc{BCD}=\struc{BC}$. Consider $x$ and $y$ in 
$\Gamma_k$ with $d^{\mathcal{F}_k}(x,y)=n+1$. By the definition of a free 
$n$-completion, the unique path $p_{x,y}$ of length $n-1$ between $x$ and $y$ 
will be contained in $\mathcal{F}_{k+1}$, whence 
$\Gamma_{k+1}\subseteq \mathcal{F}_{k+1}$. 
As $\struc{BCD}\subset \mathcal{F}_{k}$ and $p_{x,y}\cap\mathcal{F}_k=\{x,y\}$, 
we have 
\begin{eqnarray*}
 p_{x,y}\cap\struc{BCD}&= &\{x,y\}\cap\struc{BCD}\\
 &\subseteq &\Gamma_k\cap\struc{BCD}\\
 &=&\struc{BC},
\end{eqnarray*}
and thus $\Gamma_{k+1}\cap\struc{BCD}=\struc{BC}$, as desired. 

Now, note that for arbitrary graphs $X$ and $Y$, whenever 
$X\leq_nY$, then $X\cap U\leq_nY\cap U$ for any
$U\subseteq \Gamma_k$. Thus, for
$U=\struc{ABC}$, the equations (\ref{eq:freeamalg}) and (\ref{eq:intersec}) 
from above imply
\begin{equation*} 
\struc{AC}\amalgam_{C}\struc{BC}\overset{(4)}{=}\left(\struc{AC}
\amalgam_{C}\struc{BCD} \right)
\cap\struc{ABC}\overset{(3)}{\leq_n}\struc{ABCD}
\cap\struc{ABC}=\struc{ABC},
\end{equation*}
so we got $A\ind_CB$, as desired.

It remains to show that $A\ind_{\struc{BC}}D$. Again by the equality in 
(\ref{eq:intersec}), we know that $\struc{ABC}\amalgam_{\struc{BC}}\struc{BCD}$ 
is
contained in $\struc{ABCD}$. Furthermore, the independence $A\ind_CB$ implies
that
$\struc{ABC}\amalgam_{\struc{BC}}\struc{BCD}$ arises from
$\struc{AC}\amalgam_{C}\struc{BCD}$ by successively patching new paths of length
$n-1$. As $\struc{AC}\amalgam_{C}\struc{BCD}\leq_n\struc{ABCD}$, Lemma 
\ref{infinitedelta}.ii) yields
$\struc{ABC}\amalgam_{\struc{BC}}\struc{BCD}\leq_n\struc{ABCD}$ and thus
$A\ind_{BC}D$ as desired.

\textit{Stationarity:} Assume $A_1,A_2,B$ and $C$ given such that $A_1$ and 
$A_2$
have
the same type over $C$ and are both independent from $B$
over $C$. Then in particular, the graphs $\struc{A_iC}$ and $\struc{BC}$ form a 
free
amalgam over $C$,
whence there is partial isomorphism
$\varphi:\samalgamo{A_1}{C}{B}\rightarrow \samalgamo{A_2}{C}{B}$ sending $A_1$
to
$A_2$ while fixing $\struc{BC}$. Clearly, the map $\varphi$ extends to the free 
completions $\tilde{\varphi}:\struc{A_1BC}\rightarrow \struc{A_2BC}$. As 
$\Gamma_n$ is 
homogeneous, the partial isomorphism $\tilde{\varphi}$ extends to an 
automorphism of 
$\Gamma_n$, which still fixes $\struc{BC}$ and maps $A_1$ to $A_2$. Thus, the 
substructures $A_1$ and $A_2$ have the same type over $\struc{BC}$. 
\end{proof}

\begin{cor}
 The automorphism group of the countable homogeneous universal generalized
$n$-gon $\Gamma_n$ is universal for the class of all automorphism groups of
generalized $n$-gons that are free over a finite subset.
\end{cor}

\bibliography{diplom2}
\end{document}